\documentclass[leqno,12pt]{amsart} 
\usepackage[top=30truemm,bottom=30truemm,left=25truemm,right=25truemm]{geometry}
\usepackage{amssymb}
\usepackage{amsmath}
\usepackage{amsthm}
\usepackage{amscd}
\usepackage{mathrsfs}
\usepackage{graphicx}
\usepackage[dvips]{color}
\usepackage[all]{xy}

\usepackage{url}
\usepackage{comment}
\theoremstyle{plain} 
\newtheorem{theorem}{\indent\sc Theorem}[section]
\newtheorem{lemma}[theorem]{\indent\sc Lemma}
\newtheorem{corollary}[theorem]{\indent\sc Corollary}
\newtheorem{proposition}[theorem]{\indent\sc Proposition}

\theoremstyle{definition} 
\newtheorem{definition}[theorem]{\indent\sc Definition}

\newtheorem{example}[theorem]{\indent\sc Example}

\begin{document}
\title[Geodesics of toric psh functions]{Local geodesics between toric plurisubharmonic functions with infinite energy} 

\author[G. Hosono]{Genki Hosono} 

\subjclass[2010]{ 
Primary 32U05; Secondary 32U25. 
}
%
\keywords{ 
Weak geodesics, toric plurisubharmonic functions.
}
\address{
Graduate School of Mathematical Sciences, The University of Tokyo \endgraf
3-8-1 Komaba, Meguro-ku, Tokyo, 153-8914 \endgraf
Japan
}
\email{genkih@ms.u-tokyo.ac.jp}
\maketitle
\begin{abstract}
Our interest is the behavior of weak geodesics between two plurisubharmonic functions on pseudoconvex domains.
We characterize the convergence condition along the geodesic between toric psh functions with a pole at origin on a unit ball in $\mathbb{C}^n$ by means of Lelong numbers.
\end{abstract}

\section{Introduction}
In \cite{Mab}, Mabuchi defined a metric on the space of K\"{a}hler potentials on a compact complex manifold.
After that, the geodesics between K\"{a}hler potentials are studied.
By Donaldson (\cite{Donaldson}) and Semmes (\cite{Semmes}), the geodesics are written as a solution of a complex Monge-Amp\`{e}re equation.
To study them, weak geodesics are introduced as an envelope of functions with appropriate conditions (for more details, see \cite{Darvas}).
As an analogue on a pseudoconvex domain, in \cite{Rashgeod}, weak geodesics in the space of plurisubharmonic functions ({\it psh} functions for short) are studied.
Since weak geodesics are defined in a pluripotential theoretic manner, it is natural to ask how it behaves on a pseudoconvex domain.

Here we give some definitions.
\begin{definition}\label{defgeodesic}
For two negative psh functions $u_0$ and $u_1$ on $\Omega$, we define the function $\hat{u}$ on $\Omega \times S$ by
\[\hat{u}:=\sup\left\{\hat{v}\in PSH(\Omega \times S): \hat{v} \leq 0, \limsup_{\log|\zeta| \to j} \hat{v}(z,\zeta) \leq u_j(z), j=0,1 \right\}.\]
We denote $$u_t:=\hat{u}(\cdot,e^t) \text{ for } 0<t<1$$ and we call this family of psh functions as a (weak) {\it geodesic}.
In this paper, we merely call it a geodesic.
\end{definition}

The geodesics $\{u_t\}$ have good properties when the endpoints $u_j$ are psh functions with mild singularities, specifically in the finite energy class $\mathcal{F}_1(\Omega)$.
For example, functions on the weak geodesic are convergent at its endpoints and the energy functional is linear along geodesics.
We note that psh functions in $\mathcal{F}_1(\Omega)$ do not have positive Lelong numbers.
On the other hand, we also know that the geodesics between functions which may have worse singularities have worse behavior.
In particular, the geodesic between 0 and a function with positive Lelong number has discontinuity at its endpoint.
Our interest is to analyze the behavior of the geodesic between two psh functions more closely.
In this paper, we describe the behavior of the geodesics between toric psh functions with poles at the origin.
We say that a psh function is {\it toric} if it depends only on $|z_1|,\ldots,|z_n|$. 
Our main theorem is as follows.

\begin{theorem}\label{geodlel}
Let $u_0$ and $u_1$ be psh functions on the unit ball in $\mathbb{C}^n$.
Assume that $u_0$ and $u_1$ depend only on $|z_1|,\ldots,|z_n|$, $u_j^{-1}(-\infty) \subset \{0\}$ and $\lim_{z\to \zeta}u_j(z)=0$ for every $\zeta \in \partial \Omega$ and each $j=0,1$.
Let $\{u_t\}_{0<t<1}$ be a geodesic between $u_0$ and $u_1$.
Then $u_t \to u_0$ holds in capacity as $t\to 0$ if and only if $\nu(u_0\circ \phi,0) \geq \nu(u_1 \circ \phi,0)$ holds for the every curve of the form $\phi: \zeta \mapsto (a_1 \zeta^{b_1},\ldots, a_n \zeta^{b_n} )$, where $a_i \in \mathbb{C}^*$ and $b_i \in \mathbb{Z}_{>0}$.
\end{theorem}

Here, $u \to 0$ {\it in capacity} means that ${\rm Cap}(\{|u|>\epsilon \})\to 0$ for every $\epsilon > 0$.
The {\it capacity} of a Borel subset $E \subset B$ is defined as
\[{\rm Cap}(E):=\sup\left\{\int_E (dd^c u)^n : u \in PSH(B), -1\leq u \leq 0\right\}.\]

As a special case, we can describe behavior of the geodesics between toric subharmonic functions on unit disc.
In this case we only have to check Lelong numbers of $u_0$ and $u_1$ at the origin.

\begin{corollary}
Let $u_0$ and $u_1$ be toric psh functions on a unit disc in $\mathbb{C}$.
We assume that $u_0$ and $u_1$ have zero boundary value.
Then $u_t  \to u_0$ in capacity as $t \to 0$ if and only if $\nu(u_0,0)\geq \nu(u_1,0)$.
\end{corollary}

By using this corollary, we can construct an example of a geodesic $\{u_t\}$ such that $u_0=0$, $u_1$ has infinite energy, and $u_t \to u_0$ in capacity (see Example \ref{example}).

To prove the main theorem, we use techniques developed in \cite{Darvas}, in which the behavior of geodesics in K\"{a}hler case is studied.
We also use the properties of toric psh functions heavily.
Most of them are developed in \cite{Gue}.
The general psh case of Theorem \ref{geodlel} is a problem to be solved.

The organization of the paper is as follows.
In Section 2, we review the definitions and properties of geodesics between psh functions on a pseudoconvex domain.
In Section 3, the properties of toric psh functions, associated convex functions and their convex conjugates are treated.
We prove Theorem \ref{geodlel} and make an example in Section 4 .

\subsection*{\indent Acknowledgments}
The author would like to thank his supervisor Prof.\ Shigeharu Takayama for enormous support and valuable comments.
The idea of the formulation of Theorem \ref{geodlel} is obtained from the discussion with Dr.\ Takayuki Koike.
This work is supported by the Program for Leading Graduate Schools, MEXT, Japan.
This work is also supported by JSPS KAKENHI Grant Number 15J08115.

\section{Geodesics between two psh functions}
In this section, we review the known results, especially those in \cite{Darvas} and \cite{Rashgeod}.
The geodesics between psh functions are considered in \cite{Rashgeod}.
They are defined and well-behaved for some classes of psh functions with mild singularities, but their behavior is worse for functions with strong singularities.

First we define geodesics for psh functions in Cegrell's class $\mathcal{E}_0(\Omega)$ (\cite{Ceg}) and review the results.
The class $\mathcal{E}_0(\Omega)$ is consisted of bounded psh functions $u$ on $\Omega$ with $u=0$ on $\partial \Omega$ and $\int_{\Omega} (dd^c u)^n<+\infty$.

Let $S$ be the annulus $\{\zeta\in\mathbb{C}: 1<|\zeta|<e\}$. For two psh functions $u_0, u_1\in \mathcal{E}_0(\Omega)$, we define the function $\hat{u}$ on $\Omega \times S$ and $u_t:=\hat{u}(\cdot,e^t)$ for $0<t<1$ as in Definition \ref{defgeodesic}. 
In this situation, we have that $u_t \to u_j$ uniformly as $t\to j$, for $j=0,1$ (\cite[Proposition 3.1]{Rashgeod}).
In \cite{Rashgeod}, it is also proved that the energy functional $E(u):=\int_{\Omega} (-u)(dd^c u)^n$ is linear along the geodesics.

The geodesics are also considered for functions in finite energy class $\mathcal{F}_1(\Omega)$. This is the class of psh functions $u$ on $\Omega$ such that there exists a sequence $\{u_j\}_j$ of psh functions in $\mathcal{E}_0(\Omega)$ decreasing to $u$ on $\Omega$ with $\sup_j E(u_j)<+\infty$ and $\sup_j \int_{\Omega}(dd^c u_j)^n<+\infty$.
Similarly as above, we consider the geodesic $u_t$ for functions $u_0,u_1 \in \mathcal{F}_1(\Omega)$. We note that the functions in $\mathcal{F}_1$ may have singularities. 
If $u_j(z) = -\infty$ for some $z\in \Omega$ and $j=0,1$, we have that $u_t(z)=-\infty$ for every $0<t<1$. 
Therefore we cannot expect the uniform convergence $u_t \to u_j$ for functions in $\mathcal{F}_1(\Omega)$. 
The more appropriate notion in this situation is the convergence in capacity. 
We say that the sequence of psh functions $\{u_n\}_{n \in \mathbb{N}}$ converges to a function $u$ {\it in capacity} if we have, for every $\epsilon>0$, ${\rm Cap}(\{|u_n-u|>\epsilon\}) \to 0$ as $n \to \infty$.
We have that, for $u_0,u_1\in\mathcal{F}_1(\Omega)$, $u_t \to u_j$ in capacity when $t\to j$ for $j=0,1$ (\cite[Theorem 5.2]{Rashgeod}).

For psh functions with worse singularities, this type of property does not hold.
For example, if $u_0$ has a pole with positive Lelong number, so does every $u_t$ (\cite[Section 6]{Rashgeod}).

Therefore we have that the geodesics are ``good'' for finite energy class and ``bad'' for functions with poles of positive Lelong numbers.
Then how about the functions with infinite energy, but whose Lelong number is everywhere zero?
This is our main interest in this paper.

To prove the main theorem, we use the techniques in \cite{Darvas}, in which paper the K\"{a}hler case is treated.
Some techniques in \cite{Darvas} also work in the pseudoconvex case, as the following theorem.

\begin{theorem}[{\cite[Theorem 5.2]{Darvas}}]
Let $u_0, u_1 \in PSH(\Omega)$, $u_0, u_1 \not\equiv-\infty$. 
Assume that $u_0, u_1=0$ at $\partial \Omega$. 
Let $\{u_t\}_t$ be a geodesic between $u_0$ and $u_1$. 
Then we have that $\lim_{t \to 0} u_t = u_0$ in capacity if and only if $P_{[u_1]}(u_0) = u_0$.
Here, $P_{[u_1]}(u_0):= (\lim_{c \to +\infty}P(u_0,u_1+C))^*$ where $P(u,v)$ denotes the greatest psh function $w$ satisfying $w\leq \min(u,v)$.
\end{theorem}

The proof is essentially the same as in \cite{Darvas}.
By this theorem, the condition in Theorem \ref{geodlel} can be written in terms of $P_{[u_1]}(u_0)$.

\section{Toric plurisubharmonic functions, associated convex functions, and their convex conjugates}
We say that a psh function $\phi$ on a neighborhood of 0 in $\mathbb{C}^n$ is {\it toric} if $\phi$ depends only on $|z_1|, \ldots, |z_n|$.
Using the following standard lemma, we can transform a toric psh function on a neighborhood of $0 \in \mathbb{C}^n$ into a convex function on a subset of $\mathbb{R}^n$ (\cite[Chapter I, Section 5]{DemCG}).
\begin{lemma}\label{psh-conv}
Let $u$ be a toric psh function on a neighborhood of $0\in \mathbb{C}^n$ invariant under the toric action.
We define a function $f$ in real $n$ variables as
\[f(\log|z_1|,\ldots,\log|z_n|) :=u(z_1,\ldots,z_n).\]
Then $f$ is increasing in each variable and convex.
Conversely, if $f$ is a convex function on $\mathbb{R}^n_{<0}$ increasing in each variable, then the function $u$ defined as 
\[u(z_1,\ldots,z_n):=f(\log|z_1|,\ldots,\log|z_n|)\]
is psh.
\end{lemma}

By this lemma, for a toric psh function $u$ on the unit ball, we associate the convex function on $$C_0:=\{(\log |z_1|,\ldots,\log|z_n|): |z_1|^2 + \cdots + |z_n|^2 < 1\} \subset \mathbb{R}^n.$$ 
Since we assumed that $u^{-1}(-\infty) = \{0\}$, for any $x\in C_0$ and $1\leq j \leq n$, we have 
\[\lim_{t \to -\infty} f(x_1+t, \ldots, x_{j-1}+t,x_j,x_{j+1}+t,\ldots,x_n+t) \neq -\infty.\]

As in \cite{Gue}, it is useful to consider the {\it convex conjugate} (or {\it Legendre transform}) of a function $f$ defined as follows:
\begin{definition}
Let $f: \mathbb{R}^n \to \mathbb{R} \cap \{\pm \infty\}$ be a function.
We define its convex conjugate $f^*$ by
\[f^*(x^*) := \sup\{\langle x^*, x \rangle - f(x): x \in \mathbb{R}^n \}.\]
\end{definition}

We regard a function $f$ on $C_0$ as a function on $\mathbb{R}^n$ by defining $f(x) = +\infty$ if $x\not\in C_0$.
In the following proposition, we collect some basic properties which we will use later.
\begin{proposition}
Let $f$ and $g$ be $\mathbb{R} \cup \{\pm\infty \}$-valued functions on $\mathbb{R}^n$ and let $c \in \mathbb{R}$ be a constant, then:\\
(1) $f^*$ is a convex lower semicontinuous (lsc) function;\\
(2) $f^{**} \leq f$; moreover, $f^{**}$ is the greatest convex lsc function no greater than $f$;\\
(3) $f\leq g \Longrightarrow f^*\geq g^*$;\\
(4) $f^* = f^{***}$;\\
(5) $(\min(f,g))^* = \max(f^*,g^*)$;\\
(6) $(f+c)^* = f^* - c$;\\
(7) let $f_j$ be a pointwise decreasing sequence of functions converging to $f$, then $f^*_j \to f^*$ pointwise.
\end{proposition}

\begin{proof}
(1) and (2) are standard facts (see, for example, \cite[Chapter 12]{Rockafellar}).
The remaining statements can be proved easily, and we only give the proof of (7).

Let $f_j$ be a pointwise decreasing sequence.
Then we have $\lim_{j\to \infty}f_j^* \leq f^*$ by (3).
To prove equality, let $x^* \in \mathbb{R}^n$.
Then, by definition of $f^*$, we can take $x_0$ such that $\langle x^*,x_0 \rangle - f(x_0) $ is arbitrarily close to $f^*(x^*)$.
We have that $\langle x^*,x_0 \rangle - f_j(x_0) \to \langle x^*,x_0 \rangle - f(x_0)$ as $j \to \infty$.
Choosing sufficiently large $j$, we can take $x_0$ and $j$ such that $\langle x^*,x_0 \rangle - f_j(x_0)$ is arbitrarily close to $f^*(x^*)$.
Since $\langle x^*,x_0 \rangle - f_j(x_0) \leq f_j^*(x^*)$, we have that $\lim_{j\to\infty} f_j^*(x^*) = f^*(x^*)$.
\end{proof}

Here we introduce the following notation for the convex envelope:

\begin{definition}
Let $C$ be a convex subset of $\mathbb{R}^n$ and let $f$ and $g$ be upper semi-continuous functions on $C$.
We define a convex envelope of $f$ as
\[P_C(f):= \sup\{u: u \text{ is a convex function on } C \text{ such that } u\leq f \}.\]
We also define as $P_C(f,g):=P_C(\min(f,g))$.
If the set $C$ is clear from the context, we denote by $P(f)$.
\end{definition}

By the preceding proposition, these functions can be written as $P(f) = f^{**}$ and $P(f,g) = (\min(f,g))^{**} = (\max(f^*,g^*))^*$.

Following \cite{Gue}, for a convex function, we introduce the function representing the ``slope'' at infinity and the subset of $\mathbb{R}^n$ called Newton convex body (we note that we use different conventions of signature).

\begin{definition}\label{defhat}
Let $f$ be a convex function on $C_0$ increasing in each variable. Fix $a\in C_0$. \\
(1) We define a function $\hat{f}$ on $\mathbb{R}_{\leq 0}^n \setminus\{0\}$ by
\[\hat{f}(w) := \lim_{t \to +\infty} \frac{f(a +tw)}{t}.\]
(2) We define the Newton convex body $P(f)$ associated to $f$ by
\[\Gamma(f):= \{\lambda \in \mathbb{R}^n: \langle \lambda ,\cdot\rangle \leq f + O(1) \}.\]
\end{definition}
Similarly we define $\Gamma(\hat{f})$.
By the definition, $\Gamma(f)$ equals to the domain of the convex conjugate $f^*$, i.e.\ the set $\{x^* : f^*(x^*) \in \mathbb{R}\}$.
We have the following relation between $\hat{f}$ and $\Gamma(f)$.

\begin{lemma}[{\cite[Lemma 1.19]{Gue}}]\label{newtonbody}
Let $f$ and $a$ be as in Definition \ref{defhat}. Then,
\[ \overline{\Gamma(f)} =\Gamma(\hat{f}) = \{\lambda \in \mathbb{R}^n: \langle \lambda, \cdot\rangle \leq \hat{f} \}.\]
\end{lemma}

In the next section, we use these facts to prove the main theorem.

\section{Proof of the main theorem}
First we rephrase the statement of the main theorem using convex functions on $C_0:=\{(\log |z_1|,\ldots,\log|z_n|): |z_1|^2 + \cdots + |z_n|^2 < 1\} \subset \mathbb{R}^n$.
We need to prove the following proposition.

\begin{proposition}\label{rephrase}
Let $f,g$ be convex functions increasing in each variable, such that $f=g=0$ on $\partial C_0$ and $f,g$ satisfies the pole condition after Lemma \ref{psh-conv}.
Then, the following two conditions are equivalent:\\
(1) $\lim_{c\to +\infty}P_{C_0}(f,g+c) = f$ a.e.\ on $C_0$;\\
(2) for each $a=(a_1,a_2,\ldots,a_n) \in C_0$ and $b=(b_1,\ldots,b_n) \in \mathbb{N}^n$, we have that 
\[\lim_{t \to -\infty} \frac{f(a+tb)}{t} \geq \lim_{t \to -\infty} \frac{g(a+tb)}{t}.\]
\end{proposition}

Here, ``a.e.'' in the condition (1) is not necessary. Indeed, the left-hand side is an increasing sequence of convex functions, thus the limit is also convex.
If two convex functions on $C_0$ coincide outside a subset of Lebesgue measure zero, they are equal on all of $C_0$.

\subsection{1-dimensional case}
First we prove the 1-dimensional case.
In this case, we have an elementary proof.

\begin{proposition}\label{1dim}
Let $f$ and $g$ be negative convex increasing functions on $\mathbb{R}_{<0}$.
Then we have that $\lim_{c \to \infty} P(f,g+c)=f$ on $\mathbb{R}$ if and only if $\lim_{t \to -\infty} f(t)/t \geq \lim_{t \to -\infty} g(t)/t$.
\end{proposition}

\begin{proof}
First we assume that $\lim_{c \to \infty} P(f,g+c)=f$.
Adding a constant to $f$ and $g$, we can also assume that $f(0) = 0$.
We note that, under this assumption, we have that $P(f,g+c)(0)=0$ for sufficiently large $c$ (since $g$ is bounded below by a linear function).
Take an arbitrary $M_0>0$.
By assumption, we have that $P(f,g+c)(M_0) \to f(M_0)$ when $c \to \infty$. 
This is an increasing limit, thus we can choose $c_0>0$ satisfying $P(f,g+c_0)(M_0) \geq f(M_0) -\epsilon$.
Then, for $M_1 > M_0$, we have that
\[P(f,g+c_0)(-M_1) \geq \frac{M_1}{M_0} P(f,g+c_0) (-M_0)\]
by $P(f,g+c)(0)=0$ and the convexity of $P$.
From these inequalities and $P(f,g+c_0) \leq g+c_0$, we have that
\[\frac{g(-M_1)}{-M_1} + \frac{c_0}{-M_1} \leq \frac{P(f,g+c_0)(-M_1)}{-M_1} \leq \frac{P(f,g+c_0)(-M_0)}{-M_0} \leq \frac{f(-M_0)}{-M_0} - \frac{\epsilon}{-M_0}.\]
Thus, taking $M_1 \to \infty$, we have that
\[\lim_{M\to\infty} \frac{g(-M)}{-M} \leq \frac{f(-M_0)}{-M_0} - \frac{\epsilon}{-M_0} \leq \frac{f(-M_0)}{-M_0}\]
for each $M_0$.
We have the desired inequality when we take $M_0 \to \infty$. 

Conversely, we assume $\lim_{t \to -\infty} f(t)/t \geq \lim_{t \to -\infty} g(t)/t$.
The inequality $\lim P(f,g+c) \leq f$ is trivial. We want to show the opposite inequality.
By assumption, we have that $f(-M)/(-M) + \epsilon > g(-M)/M$ for sufficiently large $M$, namely when $M>M_0$.
Equivalently we have $f(-M) - \epsilon M < g(-M)$ for $M>M_0$.
Then, by convexity of $f$ and $g$, it follows that $f(t)+\epsilon t < g(t) + C$ for every $t<0$ and sufficiently large $C>0$.
We also have that $f(t)+\epsilon t < f(t)$ trivially, thus
\[f(t)+ \epsilon t \leq P(f,g+C) \leq \lim_{c \to \infty}P(f,g+c).\]
Since this inequality holds for every $\epsilon>0$, we have the conclusion.
\end{proof}

\subsection{Higher dimensional case}
First we prove $(1) \Longrightarrow (2)$ in Proposition \ref{rephrase}.

\begin{proof}[Proof of $(1) \Longrightarrow (2)$]
Take $a$ and $b$ as in (2).
We denote by $l$ the half line $\{a+tb: t<0\}$.
Then we have that
\[P_{C_0}(f,g+c) \leq P_l(f,g+c)\]
on $l$, because the left-hand side is a convex function on $l$ no greater than $f$ and $g+c$.
By the fact that equality (1) holds everywhere in $C_0$, we have that $\lim_{c\to +\infty} P_l(f,g+c) = f$ on $l$.
Thus we can apply the 1-dimensional case for $f|_l$ and $g|_l$ and get the conclusion.
\end{proof}

Next we prove the opposite implication.
We shall prove that the condition (2) implies the relation of Newton convex bodies.

\begin{lemma}
Under the condition (2), we have that
\[\overline{\Gamma(f)} \subset \overline{\Gamma(g)}.\]
\end{lemma}

\begin{proof}
The condition (2) means that $\hat{f}(w) \geq \hat{g}(w)$ for every $w \in \mathbb{Q}^n_{<0}$ (for a fixed $a \in C_0$).
By convexity of $\hat{f}$ and $\hat{g}$, this inequality holds for every $w \in \mathbb{R}^n_{<0}$.
Thus it follows from Lemma \ref{newtonbody}.
\end{proof}

\begin{proof}[Proof of $(2) \Longrightarrow (1)$]
We have that $P(f,g+c) = (\max(f^*,g^*-c))^*$ and $f = f^{**}$ (on $C_0$) thus it is sufficient to prove that $\lim_{c\to+\infty} \max(f^*,g^*-c) = f^*$ pointwise because convex conjugate is continuous along the decreasing sequence.
We have that $f^*$ is finite on the interior of $\overline{\Gamma(f)}$.
Therefore, from the preceding lemma, $g^*$ is finite on the interior of $\Gamma(f)$.
It follows that $\max(f^*,g^*-c) \downarrow f^*$ on the interior of $\Gamma(f)$.
Let $Q:=\lim_{c \to +\infty}{\max(f^*,g^*-c)}$.
We shall show that $Q^* = f^{**}$. By the relation $Q^* = Q^{***}$, it is sufficient to show that $f^* = Q^{**}$.
We have that $Q \geq f^*$, $f^*$ is lower semi-continuous, and $Q=f^*$ on ${\rm int}(\Gamma(f))$.
We also have that $Q=f^* = +\infty$ on $\mathbb{R}^n \setminus \Gamma(f)$.
Therefore, $Q$ and $f^*$ are equal outside $\partial \Gamma(f)$.
Since $Q$ is convex, $Q^{**}$ is obtained as a lower semi-continuous regularization of $Q$, thus we have $Q^{**} = f^*$ and the proof is finished.
\end{proof}

\subsection{An example}
\begin{example}\label{example}
We consider the following subharmonic function 
\[u^\alpha(z):= -(-\log |z|)^\alpha\]
on a disc $\Delta_{1-\epsilon}:=\{|z|<1-\epsilon\}$ in $\mathbb{C}$ for $0<\alpha\leq 1$.
The Lelong number is $\nu(u^\alpha,0) = 0$ when $\alpha<1$ and $\nu(u^\alpha,0)=1$ when $\alpha = 1$.
We consider a geodesic between $u_0 = 0$ and $u_1 = u^\alpha$.
Then, by Theorem \ref{geodlel}, $u_t \to 0$ in capacity if and only if $\alpha<1$.

By a straightforward computation, we have that the $L^1$-energy of $u^\alpha$ defined by $E_1(u):=\int (-u^\alpha) dd^c u^\alpha$ is finite when $\alpha<1/2$ and infinite when $\alpha \geq 1/2$.
Thus we have an example also in the infinite energy case.
\end{example}

\bibliographystyle{plain}

\end{document}